\documentclass[11pt,a4paper]{article}

\usepackage{inputenc}
\usepackage{amsmath}
\usepackage{bm}
\usepackage{bbold}
\usepackage{amsthm}
\usepackage{enumerate}

\usepackage[hyphens]{url}

\usepackage{hyperref}
\usepackage{breakurl}

\usepackage{tikz}\tikzset{x=1cm,y=1cm,z=1cm}
\usetikzlibrary{perspective}

\usepackage{pgfplots}\pgfplotsset{compat=1.16}

\title{Using matrix sparsification to solve tropical linear vector equations\thanks{
Relational and Algebraic Methods in Computer Science, LNCS 14787, Springer, 2024, 193-206.}
}

\author{N. Krivulin\thanks{Faculty of Mathematics and Mechanics, Saint Petersburg State University, 28 Universitetsky Ave., St.~Petersburg, 198504, Russia, nkk@math.spbu.ru.}
}

\date{}

\newtheorem{theorem}{Theorem}
\newtheorem{lemma}[theorem]{lemma}

\theoremstyle{definition}

\begin{document}

\maketitle

\begin{abstract}
A linear vector equation in two unknown vectors is examined in the framework of tropical algebra dealing with the theory and applications of semirings and semifields with idempotent addition. We consider a two-sided equation where each side is a tropical product of a given matrix by one of the unknown vectors. We use a matrix sparsification technique to reduce the equation to a set of vector inequalities that involve row-monomial matrices obtained from the given matrices. An existence condition of solutions for the inequalities is established, and a direct representation of the solutions is derived in a compact vector form. To illustrate the proposed approach and to compare the obtained result with that of an existing solution procedure, we apply our solution technique to handle two-sided equations known in the literature. Finally, a computational scheme based on the approach to derive all solutions of the two-sided equation is discussed.
\\

\textbf{Keywords:} idempotent semifield, two-sided linear vector equation, sparsified matrix, row-monomial matrix, complete solution.
\\

\textbf{MSC (2020):} 15A80, 15A06
%15A80  	Max-plus and related algebras
%15A06  	Linear equations (linear algebraic aspects)
\end{abstract}

\section{Introduction}

In this paper, we examine a vector equation that is represented in the form
\begin{equation*}
\bm{A}\bm{x}
=
\bm{B}\bm{y},
\end{equation*}
where $\bm{A}$ and $\bm{B}$ are given matrices, $\bm{x}$ and $\bm{y}$ are unknown vectors, and the matrix-vector multiplication is defined in the tropical algebra setting. The equation has the unknowns on both sides, and therefore it is often referred to as two-sided. We note that this equation can be transformed by some simple algebra into an equivalent two-sided equation with the same unknown vector on both sides, which sufficiently  extends the class of problems under consideration. 

In the framework of tropical algebra, which deals with the theory and applications of idempotent semirings and semifields \cite{Golan2003Semirings,Heidergott2006Maxplus,Gondran2008Graphs,Butkovic2010Maxlinear,Maclagan2015Introduction}, the solution of the two-sided equation is a challenging problem not completely solved. Since the first works of P.~Butkovi\v{c} \cite{Butkovic1978Certain,Butkovic1981Solution,Butkovic1984Elimination} on the problem, the solution of the equation is still a topic of interest (see more recent papers such as \cite{Gaubert2009Tropical,Lorenzo2011Algorithm,Jones2019Twosided,Krivulin2023Solution}). This equation along with the two-sided inequality $\bm{A}\bm{x}\leq\bm{B}\bm{y}$ and (nonhomogeneous) equation $\bm{A}\bm{x}\oplus\bm{c}=\bm{B}\bm{y}\oplus\bm{d}$ find application in a range of areas in mathematics and computer science from geometry of tropical polyhedral cones \cite{Zimmermann1977General,Gaubert2009Tropical,Allamigeon2013Computing} to mean payoff games \cite{Akian2012Tropical,Gaubert2012Tropical} and machine scheduling \cite{Butkovic1984Elimination,Jones2019Twosided}.  

Several approaches to handle two-sided equations are known in the literature (see an overview of the state-of-art on the solution techniques given in Chapter~7 of \cite{Butkovic2010Maxlinear}). Existing approaches usually employ computational procedures based on iterative algorithms to find a solution or to indicate that the equation has no solutions \cite{Butkovic1981Solution,Butkovic1984Elimination,Butkovic1984Properties,Walkup1998General,Cuninghamegreen2001Equation,Butkovic2006Strongly,Lorenzo2011Algorithm,Jones2019Twosided,Krivulin2023Solution}. However, these approaches do not offer a direct analytical solution, which is of theoretical interest and particular importance.

In an attempt to move towards constructing an analytical solution, we propose a new solution scheme, which is based on the matrix sparsification technique developed to solve tropical optimization problems and two-sided inequalities in \cite{Krivulin2015Soving,Krivulin2017Algebraic,Krivulin2017Complete,Krivulin2018Complete,Krivulin2018Solution,Krivulin2020Complete}. We use this technique to reduce the two-sided equation to a set of vector inequalities that involve row-monomial matrices obtained from the given matrices. An existence condition of solutions for the inequalities is established, and a direct representation of the solutions is derived in a compact vector form. To illustrate the approach and to compare the result with those of existing solution procedures, we apply our solution technique to handle two-sided equations known in the literature. Finally, a computational scheme based on the approach to derive all solutions of the two-sided equation is discussed.

The paper is organized as follows. In Section~\ref{S-BDPR}, we give an overview of basic definitions and preliminary results of tropical algebra, which underlie subsequent developments. Section~\ref{S-VI} offers a solution of a system of vector inequalities, which is a key ingredient in the solution of the two-sided equation. In Section~\ref{S-STSE}, we present our main results, which include a necessary and sufficient condition for solutions of the equation and represent a direct solution of the equation in a compact vector form. Numerical examples are given in Section~\ref{S-NE}. Section~\ref{S-C} presents some concluding remarks.

\section{Basic Definitions and Preliminary Results}
\label{S-BDPR}
We start with basic definitions and preliminary results of tropical (idempotent) algebra that offer a formal framework for the development of solution methods for the two-sided equation. For more details on the theory and applications of tropical algebra, one can consult recent monographs and textbooks \cite{Golan2003Semirings,Heidergott2006Maxplus,Gondran2008Graphs,Butkovic2010Maxlinear,Maclagan2015Introduction}.

\subsection{Idempotent Semifield}

Let $\mathbb{X}$ be a set that is closed under associative and commutative operations $\oplus$ (addition) and $\otimes$ (multiplication), which have distinct neutral elements $\mathbb{0}$ (zero) and $\mathbb{1}$ (one) respectively. Suppose that $(\mathbb{X},\oplus,\mathbb{0})$ is an idempotent commutative monoid, $(\mathbb{X}\setminus\{\mathbb{0}\},\otimes,\mathbb{1})$ is an Abelian group, and multiplication distributes over addition. Under these conditions, the algebraic system $(\mathbb{X},\oplus,\otimes,\mathbb{0},\mathbb{1})$ is commonly referred to as an idempotent semifield.

In this semifield, addition is idempotent ($x\oplus x=x$ for all $x$), and multiplication is invertible (for each $x\ne\mathbb{0}$, there exists $x^{-1}$ such that $x\otimes x^{-1}=\mathbb{1}$). In the sequel, we drop the multiplication sign $\otimes$ to save writing. The power notation with integer exponents indicates repeated multiplication defined by $x^{0}=\mathbb{1}$, $x^{p}=xx^{p-1}$, $x^{-p}=(x^{-1})^{p}$ and $\mathbb{0}^{p}=\mathbb{0}$ for all nonzero $x$ and integer $p\geq1$.

Idempotent addition induces a partial order such that $x\leq y$ if and only if $x\oplus y=y$. With respect to this order, addition satisfies an extremal property (the majority law), which means that the inequalities $x\leq x\oplus y$ and $y\leq x\oplus y$ hold for all $x,y$. Both addition and multiplication are isotone (the inequality $x\leq y$ results in the inequalities $x\oplus z\leq y\oplus z$ and $xz\leq yz$). Inversion is antitone ($x\leq y$ yields $x^{-1}\geq y^{-1}$ for nonzero $x,y$). Finally, the inequality $x\oplus y\leq z$ is equivalent to the system of the inequalities $x\leq z$ and $y\leq z$.

Moreover, the idempotent semifield is assumed to be selective ($x\oplus y\in\{x,y\}$ for all $x,y$), which turns the above partial order into a total order.

An example of the idempotent semifield under consideration is the real semifield $\mathbb{R}_{\max,+}=(\mathbb{R}\cup\{-\infty\},\max,+,-\infty,0)$ also called $(\max,+)$-algebra. In this semifield, addition $\oplus$ is defined as $\max$, multiplication $\otimes$ as arithmetic addition $+$, zero $\mathbb{0}$ as $-\infty$, and one $\mathbb{1}$ as $0$. For each $x\ne-\infty$, there exists the inverse $x^{-1}$, which coincides with the opposite number $-x$ in the conventional algebra. The power $x^{y}$ corresponds to the ordinary product $x\times y$. The order, which is associated with addition, corresponds to the natural linear order on $\mathbb{R}$.

The system $\mathbb{R}_{\min,\times}=(\mathbb{R}_{+}\cup\{+\infty\},\min,\times,+\infty,1)$ with $\mathbb{R}_{+}=\{x\in\mathbb{R}|\ x>0\}$ is another example known as $\min$-algebra. It is equipped with $\oplus=\min$, $\otimes=\times$, $\mathbb{0}=+\infty$ and $\mathbb{1}=1$. Inversion and exponentiation are defined in the usual way. The order induced by addition is opposite to the natural linear order on $\mathbb{R}$.

\subsection{Matrices and Vectors}

The algebra of matrices with entries from an idempotent semifield is routinely defined. The set of matrices over $\mathbb{X}$ which consist of $m$ rows and $n$ columns is denoted by $\mathbb{X}^{m\times n}$. A matrix with all entries equal to $\mathbb{0}$ is the zero matrix denoted by $\bm{0}$. A matrix that has no rows with all entries equal to $\mathbb{0}$ is called row-regular.

Any matrix of one column (row) is a column (row) vector. The set of column vectors with $n$ elements over $\mathbb{X}$ is denoted by $\mathbb{X}^{n}$. All vectors below are column vectors unless otherwise specified. A vector with all elements equal to $\mathbb{0}$ is the zero vector denoted by $\bm{0}$. A vector without elements equal to $\mathbb{0}$ is called regular.

Matrix operations follow the standard entrywise computation schemes where arithmetic addition and multiplication are replaced by $\oplus$ and $\otimes$. For any conforming matrices $\bm{A}=(a_{ij})$, $\bm{B}=(b_{ij})$, $\bm{C}=(c_{ij})$ and scalar $x$, matrix addition, matrix multiplication and multiplication by scalar are given by
\begin{equation*}
\{\bm{A}\oplus\bm{B}\}_{ij}
=
a_{ij}\oplus b_{ij},
\qquad
\{\bm{A}\bm{C}\}_{ij}
=
\bigoplus_{k}a_{ik}c_{kj},
\qquad
\{x\bm{A}\}_{ij}
=
xa_{ij}.
\end{equation*}

The monotonicity of scalar operations and other properties associated with the order relation defined on $\mathbb{X}$ are extended to matrices, where the relations for matrices are understood componentwise. 

For any nonzero matrix $\bm{A}=(a_{ij})\in\mathbb{X}^{m\times n}$, the multiplicative inverse transpose (or the conjugate) is the matrix $\bm{A}^{-}=(a_{ij}^{-})\in\mathbb{X}^{n\times m}$ with the entries given by the condition: $a_{ij}^{-}=a_{ji}^{-1}$ if $a_{ji}\ne\mathbb{0}$, and $a_{ij}^{-}=\mathbb{0}$ otherwise. 

Consider the set $\mathbb{X}^{n\times n}$ of square matrices of order $n$. A matrix with the entries equal to $\mathbb{1}$ on the diagonal and to $\mathbb{0}$ elsewhere is the identity matrix denoted $\bm{I}$. For any matrix $\bm{A}\in\mathbb{X}^{n\times n}$ and integer $p>0$, the matrix powers are defined as $\bm{A}^{0}=\bm{I}$ and $\bm{A}^{p}=\bm{A}\bm{A}^{p-1}$. The trace of the matrix $\bm{A}=(a_{ij})$ is given by
\begin{equation*}
\operatorname{tr}\bm{A}
=
a_{11}\oplus a_{22}\oplus\cdots\oplus a_{nn}
=
\bigoplus_{i=1}^{n}a_{ii}.
\end{equation*}

The trace is monotone, which means that the inequality $\bm{A}\leq\bm{B}$ results in the inequality $\operatorname{tr}\bm{A}\leq\operatorname{tr}\bm{B}$. Further properties of trace are given by the equalities
\begin{equation*}
\operatorname{tr}(\bm{A}\oplus\bm{B})
=
\operatorname{tr}\bm{A}
\oplus
\operatorname{tr}\bm{B},
\qquad
\operatorname{tr}(\bm{A}\bm{C})
=
\operatorname{tr}(\bm{C}\bm{A}),
\qquad
\operatorname{tr}(x\bm{A})
=
x\operatorname{tr}\bm{A},
\end{equation*}
which are valid for all matrices $\bm{A}$, $\bm{B}$, $\bm{C}$ of appropriate size and scalar $x$.

For any matrix $\bm{A}\in\mathbb{X}^{n\times n}$, a trace function is defined that maps $\bm{A}$ to
\begin{equation*} 
\operatorname{Tr}(\bm{A})
=
\operatorname{tr}\bm{A}
\oplus
\operatorname{tr}(\bm{A}^{2})
\oplus\cdots\oplus
\operatorname{tr}(\bm{A}^{n})
=
\bigoplus_{i=1}^{n}\operatorname{tr}(\bm{A}^{i}).
\end{equation*} 

This function is monotone since the inequality $\bm{A}\leq\bm{B}$ yields the inequality $\operatorname{tr}(\bm{A}^{i})\leq\operatorname{tr}(\bm{B}^{i})$ for each integer $i\geq0$, and therefore $\operatorname{Tr}(\bm{A})\leq\operatorname{Tr}(\bm{B})$. It retains the cyclic property of trace, which implies that $\operatorname{Tr}(\bm{A}\bm{B})=\operatorname{Tr}(\bm{B}\bm{A})$.

The Kleene star operator takes the matrix $\bm{A}$ to the Kleene star matrix
\begin{equation*} 
\bm{A}^{\ast}
=
\bm{I}\oplus\bm{A}\oplus\bm{A}^{2}
\oplus\cdots
=
\bigoplus_{i\geq0}\bm{A}^{i}.
\end{equation*} 

Suppose that the condition $\operatorname{Tr}(\bm{A})\leq\mathbb{1}$ holds. It is not difficult to verify that under this condition, the Kleene star matrix becomes the finite sum (we assume below that the condition is satisfied whenever a Kleene matrix is defined):
\begin{equation*} 
\bm{A}^{\ast}
=
\bm{I}\oplus\bm{A}\oplus\cdots\oplus\bm{A}^{n-1}
=
\bigoplus_{i=0}^{n-1}\bm{A}^{i}.
\end{equation*} 

Moreover, since $\bm{A}^{p}\leq\bm{A}\bm{A}^{\ast}$ for $p\geq1$ and $\operatorname{tr}(\bm{A}\bm{A}^{\ast})=\operatorname{Tr}(\bm{A})$, we have 
\begin{equation*}
\operatorname{tr}(\bm{A}^{p})
\leq
\operatorname{Tr}(\bm{A}),
\qquad
p\geq1.
\end{equation*} 

We complete the overview with the next result. A row-regular matrix with exactly one nonzero entry in each row is called row-monomial. If a matrix $\bm{A}$ is row-monomial, then the following matrix inequalities are valid (which correspond to the univalent and total properties in the context of relational algebra):
\begin{equation*}
\bm{A}^{-}\bm{A}
\leq
\bm{I},
\qquad
\bm{A}\bm{A}^{-}
\geq
\bm{I}.
\end{equation*}

\section{Vector Inequalities}
\label{S-VI}

The purpose of this section is to examine a system of vector inequalities that is a key ingredient in the subsequent solution of the two-sided equation. As a preliminary result, we use the solution of the problem: given a square matrix $\bm{A}\in\mathbb{X}^{n\times n}$, find all regular solutions $\bm{x}\in\mathbb{X}^{n}$ to the inequality 
\begin{equation}
\bm{A}\bm{x}
\leq
\bm{x}.
\label{I-Axlex}
\end{equation}

A complete solution of inequality \eqref{I-Axlex} can be obtained as follows \cite{Krivulin2015Extremal,Krivulin2015Multidimensional}.
\begin{lemma}
\label{L-Axlex}
For any matrix $\bm{A}$, the following statements are true.
\begin{enumerate}
\item
If $\operatorname{Tr}(\bm{A})\leq\mathbb{1}$, then all regular solutions of inequality \eqref{I-Axlex} are given in parametric form by $\bm{x}=\bm{A}^{\ast}\bm{u}$ where $\bm{u}\ne\bm{0}$ is any parameter vector such that $\bm{x}$ is regular.
\item
If $\operatorname{Tr}(\bm{A})>\mathbb{1}$, then there is no regular solution.
\end{enumerate}
\end{lemma}

Suppose now that $\bm{A}\in\mathbb{X}^{k\times n}$ and $\bm{B}\in\mathbb{X}^{n\times k}$ are given matrices, and the problem is to find a regular solution in the form of a pair of regular vectors $\bm{x}\in\mathbb{X}^{n}$ and $\bm{y}\in\mathbb{X}^{k}$ that satisfy the system of inequalities
\begin{equation}
\bm{A}\bm{x}
\leq
\bm{y},
\qquad
\bm{B}\bm{y}
\leq
\bm{x}.
\label{I-Axley-Bylex}
\end{equation}

The next result offers a complete solution to system \eqref{I-Axley-Bylex}.
\begin{lemma}
\label{L-Axley-Bylex}
For any matrices $\bm{A}$ and $\bm{B}$, the system at \eqref{I-Axley-Bylex} has regular solutions if and only if the following condition holds:
\begin{equation*}
\operatorname{Tr}(\bm{A}\bm{B})
\leq
\mathbb{1}.
%\label{I-TrABle1}
\end{equation*}

Under this condition, all solutions are given in parametric form by
\begin{equation*}
\begin{aligned}
\bm{x}
&=
(\bm{B}\bm{A})^{\ast}\bm{u}
\oplus
\bm{B}(\bm{A}\bm{B})^{\ast}\bm{v},
\\
\bm{y}
&=
\bm{A}(\bm{B}\bm{A})^{\ast}\bm{u}
\oplus
(\bm{A}\bm{B})^{\ast}\bm{v},
\end{aligned}
%\label{E-x-y}
\end{equation*}
where $\bm{u}$ and $\bm{v}$ are any parameter vectors such that both $\bm{x}$ and $\bm{y}$ are regular.
\end{lemma}
\begin{proof}
First, we combine system \eqref{I-Axley-Bylex} into one inequality with a block matrix and a block vector. We introduce the notation
\begin{equation*}
\bm{C}
=
\left(
\begin{array}{cc}
\bm{0} & \bm{B}
\\
\bm{A} & \bm{0}
\end{array}
\right),
\qquad
\bm{z}
=
\left(
\begin{array}{c}
\bm{x}
\\
\bm{y}
\end{array}
\right),
\qquad
\bm{w}
=
\left(
\begin{array}{c}
\bm{u}
\\
\bm{v}
\end{array}
\right).
\end{equation*}

We observe that $\bm{C}$ is a square matrix of order $n+k$, where the upper right block $\bm{B}$ is of size $n\times k$ and the lower left block $\bm{A}$ is of size $k\times n$. The vectors $\bm{z}$ and $\bm{w}$ are of order $n+k$, whereas $\bm{u}$ is of order $n$ and $\bm{v}$ is of $k$. 

With the above notation, the system takes the form of the inequality
\begin{equation}
\bm{C}\bm{z}
\leq
\bm{z}.
\label{I-Czlez}
\end{equation}

We now apply Lemma~\ref{L-Axlex} to examine inequality \eqref{I-Czlez}. Calculation of the even and odd powers of the matrix $\bm{C}$ leads to the results
\begin{equation*}
\bm{C}^{2i}
=
\left(
\begin{array}{cc}
(\bm{B}\bm{A})^{i} & \bm{0}
\\
\bm{0} & (\bm{A}\bm{B})^{i}
\end{array}
\right),
\qquad
\bm{C}^{2i+1}
=
\left(
\begin{array}{cc}
\bm{0} & \bm{B}(\bm{A}\bm{B})^{i}
\\
\bm{A}(\bm{B}\bm{A})^{i} & \bm{0}
\end{array}
\right),
\qquad
i=0,1,\ldots
\end{equation*}

Next, we evaluate the traces of the powers of $\bm{C}$ and find that the odd powers have the trace $\operatorname{tr}(\bm{C}^{2i+1})=\mathbb{0}$. For even powers, the cyclic property of trace yields
\begin{equation*}
\operatorname{tr}(\bm{C}^{2i})
=
\operatorname{tr}(\bm{B}\bm{A})^{i}
=
\operatorname{tr}(\bm{A}\bm{B})^{i},
\end{equation*}
where the matrix products $\bm{B}\bm{A}$ and $\bm{A}\bm{B}$ under the trace operator form square matrices of orders $n$ and $k$ respectively.

Let $l=\lfloor(n+k)/2\rfloor$ denote the greatest integer less than or equal to $(n+k)/2$. According to Lemma~\ref{L-Axlex}, inequality \eqref{I-Czlez} has regular solutions if and only if 
\begin{equation}
\operatorname{Tr}(\bm{C})
=
\bigoplus_{i=1}^{n+k}
\operatorname{tr}(\bm{C}^{i})
=
\bigoplus_{i=1}^{l}
\operatorname{tr}(\bm{C}^{2i})
\leq
\mathbb{1}.
\label{I-TrC}
\end{equation}

Suppose that the relation $n\leq k$ is valid, which implies that $l\geq n$. Observing that the matrix $\bm{B}\bm{A}$ is of order $n$, we use condition \eqref{I-TrC} to write
\begin{equation*}
\operatorname{Tr}(\bm{B}\bm{A})
=
\bigoplus_{i=1}^{n}
\operatorname{tr}(\bm{B}\bm{A})^{i}
\leq
\bigoplus_{i=1}^{l}
\operatorname{tr}(\bm{B}\bm{A})^{i}
=
\bigoplus_{i=1}^{l}
\operatorname{tr}(\bm{C}^{2i})
=
\operatorname{Tr}(\bm{C})
\leq
\mathbb{1},
\end{equation*}
which shows that the condition $\operatorname{Tr}(\bm{B}\bm{A})\leq\mathbb{1}$ is a consequence of \eqref{I-TrC}.

Let us verify that the condition $\operatorname{Tr}(\bm{B}\bm{A})\leq\mathbb{1}$ yields condition \eqref{I-TrC}, and thus both conditions are equivalent to each other. Indeed, as a result of the former condition, the inequality $\operatorname{tr}(\bm{B}\bm{A})^{i}\leq\operatorname{Tr}(\bm{B}\bm{A})\leq\mathbb{1}$ holds for all integers $i\geq1$, which yields the latter condition as follows:
\begin{equation*}
\operatorname{Tr}(\bm{C})
=
\bigoplus_{i=1}^{l}
\operatorname{tr}(\bm{B}\bm{A})^{i}
\leq
\mathbb{1}.
\end{equation*}

Provided that the inequality $n>k$ holds, we use a similar argument to verify that the condition $\operatorname{Tr}(\bm{C})\leq\mathbb{1}$ is equivalent to $\operatorname{Tr}(\bm{A}\bm{B})\leq\mathbb{1}$.

Furthermore, it follows from the cyclic property of trace and the obtained conditions that for all integers $i\geq1$, we have
\begin{equation*}
\operatorname{tr}(\bm{B}\bm{A})^{i}
=
\operatorname{tr}(\bm{A}\bm{B})^{i}
\leq
\mathbb{1}.
\end{equation*}

Since all these traces have the same value not exceeding $\mathbb{1}$, we arrive at the equality $\operatorname{Tr}(\bm{A}\bm{B})=\operatorname{Tr}(\bm{B}\bm{A})$. As a result, we conclude that both conditions $\operatorname{Tr}(\bm{A}\bm{B})\leq\mathbb{1}$ and $\operatorname{Tr}(\bm{B}\bm{A})\leq\mathbb{1}$ are equivalent, and take the last condition as the existence condition for solutions of inequality \eqref{I-Czlez} and hence of system \eqref{I-Axley-Bylex}.

We are now in a position to derive all solutions of inequality \eqref{I-Czlez}. We apply Lemma~\ref{L-Axlex} to represent these solutions in parametric form as
\begin{equation*}
\bm{z}
=
\bm{C}^{\ast}\bm{w},
%\qquad
%\bm{C}^{\ast}
%=
%\bigoplus_{i=0}^{n+k-1}\bm{C}^{i},
%\label{E-zeqCastw}
\end{equation*}
where $\bm{w}\ne\bm{0}$ is any parameter vector such that $\bm{z}$ is regular.

We consider the Kleene star matrix $\bm{C}^{\ast}$, which generates the solutions, and note that the existence condition at \eqref{I-TrC} yields the inequality $\bm{C}^{\ast}\geq\bm{C}^{i}$, which is valid for all integers $i\geq0$. Therefore, with $l=\max(n,k)-1$, we can write
\begin{equation*}
\bm{C}^{\ast}
=
\bigoplus_{i=0}^{n+k-1}
\bm{C}^{i}
=
\bigoplus_{i=0}^{l}
(\bm{C}^{2i}
\oplus
\bm{C}^{2i+1}).
\end{equation*}

Substitution of the expressions for the powers of $\bm{C}$ leads to the matrix
\begin{equation*}
\bm{C}^{\ast}
=
\bigoplus_{i=0}^{l}
\left(
\begin{array}{cc}
(\bm{B}\bm{A})^{i} & \bm{B}(\bm{A}\bm{B})^{i}
\\
\bm{A}(\bm{B}\bm{A})^{i} & (\bm{A}\bm{B})^{i}
\end{array}
\right).
\end{equation*}

It follows from the conditions $\operatorname{Tr}(\bm{A}\bm{B})\leq\mathbb{1}$ and $\operatorname{Tr}(\bm{B}\bm{A})\leq\mathbb{1}$ that the inequalities $(\bm{A}\bm{B})^{\ast}\geq(\bm{A}\bm{B})^{i}$ and $(\bm{B}\bm{A})^{\ast}\geq(\bm{B}\bm{A})^{i}$ hold for all integers $i\geq0$. Observing that $l\geq n-1$ and $l\geq k-1$, we have the equalities
\begin{equation*}
(\bm{A}\bm{B})^{\ast}
=
\bigoplus_{i=0}^{l}
(\bm{A}\bm{B})^{i},
\qquad
(\bm{B}\bm{A})^{\ast}
=
\bigoplus_{i=0}^{l}
(\bm{B}\bm{A})^{i}.
\end{equation*}

The application of these equalities leads to the Kleene matrix in the following form (see also \cite{Sakarovitch2009Elements}):
\begin{equation*}
\bm{C}^{\ast}
=
\left(
\begin{array}{cc}
(\bm{B}\bm{A})^{\ast} & \bm{B}(\bm{A}\bm{B})^{\ast}
\\
\bm{A}(\bm{B}\bm{A})^{\ast} & (\bm{A}\bm{B})^{\ast}
\end{array}
\right).
\end{equation*}

Finally, we turn from the solution $\bm{z}$ of inequality \eqref{I-Czlez} to the solutions $\bm{x}$ and $\bm{y}$ of system \eqref{I-Axley-Bylex} to obtain the parametric representation offered by the lemma. 
\qed
\end{proof}

\section{Solution of Two-Sided Equation}
\label{S-STSE}

In this section, we apply previous results to the solution of the following problem. Given matrices $\bm{A}\in\mathbb{X}^{m\times n}$ and $\bm{B}\in\mathbb{X}^{m\times k}$, we need to find a pair of regular vectors $\bm{x}\in\mathbb{X}^{n}$ and $\bm{y}\in\mathbb{X}^{k}$ that satisfy the equation
\begin{equation}
\bm{A}\bm{x}
=
\bm{B}\bm{y}.
\label{E-AxeqBy}
\end{equation}

Our purpose is to establish an existence condition for solutions of equation \eqref{E-AxeqBy} and to derive a solution if the condition is fulfilled. We start with a lemma that provides a necessary and sufficient condition for regular solutions to exist. 
\begin{lemma}
\label{L-AxeqBy}
A pair of regular vectors $\bm{x}$ and $\bm{y}$ is a solution of equation \eqref{E-AxeqBy} with row-regular matrices $\bm{A}$ and $\bm{B}$ if and only if it is a solution of the system
\begin{equation}
\bm{B}_{1}^{-}\bm{A}\bm{x}
\leq
\bm{y},
\qquad
\bm{A}_{1}^{-}\bm{B}\bm{y}
\leq
\bm{x}
\label{I-B1Axley-A1Bylex}
\end{equation}
for row-monomial matrices $\bm{A}_{1}$ and $\bm{B}_{1}$ obtained respectively from $\bm{A}$ and $\bm{B}$ by keeping one nonzero entry in each row unchanged and setting the others to $\mathbb{0}$.
\end{lemma}
\begin{proof}
To prove the lemma, we first note that equation \eqref{E-AxeqBy} is equivalent to (and has the common solutions with) the system of the opposite vector inequalities 
\begin{equation}
\bm{A}\bm{x}
\leq
\bm{B}\bm{y},
\qquad
\bm{A}\bm{x}
\geq
\bm{B}\bm{y}.
\label{I-AxleBy-AxgeBy}
\end{equation}

Let us verify that any solution of this system and hence of equation \eqref{E-AxeqBy} is a solution of system \eqref{I-B1Axley-A1Bylex} for some row-monomial matrices $\bm{A}_{1}$ and $\bm{B}_{1}$, and vice versa. Suppose that system \eqref{I-AxleBy-AxgeBy} has a regular solution in the form of a pair of regular vectors $\bm{x}=(x_{j})$ and $\bm{y}=(y_{j})$. Consider the first vector inequality at \eqref{I-AxleBy-AxgeBy} and examine its scalar inequalities corresponding to the rows of the matrices $\bm{A}$ and $\bm{B}$. For each row $i=1,\ldots,m$, we have the scalar inequality
\begin{equation*}
a_{i1}x_{1}
\oplus\cdots\oplus
a_{in}x_{n}
\leq
b_{i1}y_{1}
\oplus\cdots\oplus
b_{ik}y_{k}.
\end{equation*}

Observing that the addition $\oplus$ is selective, we see that the sum on the right-hand side is equal to one of its summands, say to the term $b_{ij}y_{j}$ where $b_{ij}>\mathbb{0}$ due to row regularity of $\bm{B}$. Therefore, this scalar inequality implies an inequality with the right-hand side truncated to one term as $a_{i1}x_{1}\oplus\cdots\oplus a_{in}x_{n}\leq b_{ij}y_{j}$. 
  
Assume that the maximum summands are fixed on the right of the scalar inequalities for all rows. Let $\bm{B}_{1}$ be a row-monomial matrix composed from $\bm{B}$, where nonzero entries agree with the maximum summands in the rows. Then, we can combine the truncated inequalities into one vector inequality $\bm{A}\bm{x}\leq\bm{B}_{1}\bm{y}$.

Note that the matrix $\bm{B}_{1}$ is row-monomial, and thus the inequality $\bm{B}_{1}^{-}\bm{B}_{1}\leq\bm{I}$ is valid. We use this inequality to write $\bm{B}_{1}^{-}\bm{A}\bm{x}\leq\bm{B}_{1}^{-}\bm{B}_{1}\bm{y}\leq\bm{y}$, which means that the pair of vectors $\bm{x}$ and $\bm{y}$ satisfies the inequality $\bm{B}_{1}^{-}\bm{A}\bm{x}\leq\bm{y}$. By applying a similar argument, we conclude that the inequality $\bm{A}_{1}^{-}\bm{B}\bm{y}\leq\bm{x}$ is also valid for some row-monomial matrix $\bm{A}_{1}$ obtained from $\bm{A}$.

Let us now verify that any pair of vectors $\bm{x}$ and $\bm{y}$ that satisfies system \eqref{I-B1Axley-A1Bylex}, also solves system \eqref{I-AxleBy-AxgeBy}. We take the first inequality at \eqref{I-B1Axley-A1Bylex} and multiply it by $\bm{B}$ from the left to obtain $\bm{B}\bm{B}_{1}^{-}\bm{A}\bm{x}\leq\bm{B}\bm{y}$. Since $\bm{B}\geq\bm{B}_{1}$ and $\bm{B}_{1}\bm{B}_{1}^{-}\geq\bm{I}$, we have $\bm{A}\bm{x}\leq\bm{B}_{1}\bm{B}_{1}^{-}\bm{A}\bm{x}\leq\bm{B}\bm{B}_{1}^{-}\bm{A}\bm{x}\leq\bm{B}\bm{y}$, which yields the inequality $\bm{A}\bm{x}\leq\bm{B}\bm{y}$.

The opposite inequality $\bm{A}\bm{x}\geq\bm{B}\bm{y}$ is obtained in a similar way.
\qed
\end{proof}

The next theorem combines previous results into one statement which provides a convenient way to handle equation \eqref{E-AxeqBy}. One can consider some related results in the context of the analysis of tropical polyhedral cones in \cite{Develin2004Tropical,Truffet2010Decomposition}.

\begin{theorem}
\label{T-AxeqBy}
Let $\bm{A}$ and $\bm{B}$ be row-regular matrices, and $\mathcal{A}_{1}$ and $\mathcal{B}_{1}$ be the sets of row-monomial matrices obtained respectively from $\bm{A}$ and $\bm{B}$ by keeping one nonzero entry in each row unchanged and setting the others to $\mathbb{0}$.

Then, equation \eqref{E-AxeqBy} has regular solutions if and only if there exist matrices $\bm{A}_{1}\in\mathcal{A}_{1}$ and $\bm{B}_{1}\in\mathcal{B}_{1}$ such that the following condition holds:
\begin{equation}
\operatorname{Tr}(\bm{A}\bm{A}_{1}^{-}\bm{B}\bm{B}_{1}^{-})
=
\mathbb{1}.
\label{I-TrAA1BB1eq1}
\end{equation}

Under this condition, all solutions are given in parametric form by
\begin{equation}
\begin{aligned}
\bm{x}
&=
(\bm{A}_{1}^{-}\bm{B}\bm{B}_{1}^{-}\bm{A})^{\ast}\bm{u}
\oplus
\bm{A}_{1}^{-}\bm{B}(\bm{B}_{1}^{-}\bm{A}\bm{A}_{1}^{-}\bm{B})^{\ast}\bm{v},
\\
\bm{y}
&=
\bm{B}_{1}^{-}\bm{A}(\bm{A}_{1}^{-}\bm{B}\bm{B}_{1}^{-}\bm{A})^{\ast}\bm{u}
\oplus
(\bm{B}_{1}^{-}\bm{A}\bm{A}_{1}^{-}\bm{B})^{\ast}\bm{v},
\end{aligned}
\label{E-x-y}
\end{equation}
where $\bm{u}$ and $\bm{v}$ are any parameter vectors such that both $\bm{x}$ and $\bm{y}$ are regular.

Otherwise, the equation has no regular solutions.
\end{theorem}
\begin{proof}
We begin with an application of Lemma~\ref{L-AxeqBy} to reduce equation \eqref{E-AxeqBy} to the set of systems \eqref{I-B1Axley-A1Bylex}, where each system is defined by a pair of row-monomial matrices $\bm{A}_{1}\in\mathcal{A}$ and $\bm{B}_{1}\in\mathcal{B}$ obtained from $\bm{A}$ and $\bm{B}$ as specified in the theorem.

To verify whether a system has regular solutions, we use the existence condition given by Lemma~\ref{L-Axley-Bylex}, where $\bm{A}$ is replaced by $\bm{B}_{1}^{-}\bm{A}$ and $\bm{B}$ by $\bm{A}_{1}^{-}\bm{B}$, to write $\operatorname{Tr}(\bm{A}_{1}^{-}\bm{B}\bm{B}_{1}^{-}\bm{A})\leq\mathbb{1}$. Furthermore, it follows from the cyclic property of the trace function that $\operatorname{Tr}(\bm{A}_{1}^{-}\bm{B}\bm{B}_{1}^{-}\bm{A})=\operatorname{Tr}(\bm{A}\bm{A}_{1}^{-}\bm{B}\bm{B}_{1}^{-})$. Since the inequalities $\bm{A}_{1}^{-}\bm{A}\geq\bm{I}$ and $\bm{B}_{1}^{-}\bm{B}\geq\bm{I}$ always hold, and hence $\operatorname{Tr}(\bm{A}\bm{A}_{1}^{-}\bm{B}\bm{B}_{1}^{-})\geq\mathbb{1}$, the existence condition takes the form of \eqref{I-TrAA1BB1eq1}.

Finally, with the substitution of $\bm{A}$ by $\bm{B}_{1}^{-}\bm{A}$ and $\bm{B}$ by $\bm{A}_{1}^{-}\bm{B}$ in the solution provided by Lemma~\ref{L-Axley-Bylex}, leads to \eqref{E-x-y}.
\qed
\end{proof}

\section{Numerical Examples}
\label{S-NE}
We now illustrate the obtained results with a solution of the problem examined in \cite{Cuninghamegreen2003Equation} in the framework of the semifield $\mathbb{R}_{\max,+}$, also known as $(\max,+)$-algebra.

Consider equation \eqref{L-AxeqBy} with the given matrices and unknown vectors defined (with the use of the symbol $\mathbb{0}=-\infty$) as follows:
\begin{equation*}
\bm{A}
=
\left(
\begin{array}{ccc}
3 & \mathbb{0} & 0
\\
1 & 1 & 0
\\
\mathbb{0} & 1 & 2
\end{array}
\right),
\qquad
\bm{B}
=
\left(
\begin{array}{cc}
1 & 1
\\
3 & 2
\\
3 & 1
\end{array}
\right),
\qquad
\bm{x}
=
\left(
\begin{array}{c}
x_{1}
\\
x_{2}
\\
x_{3}
\end{array}
\right),
\qquad
\bm{y}
=
\left(
\begin{array}{c}
y_{1}
\\
y_{2}
\end{array}
\right).
\end{equation*}

The solution, which is obtained in \cite{Cuninghamegreen2003Equation}, takes the form
\begin{equation*}
\bm{x}
=
\left(
\begin{array}{c}
0
\\
3
\\
2
\end{array}
\right),
\qquad
\bm{y}
=
\left(
\begin{array}{c}
1
\\
2
\end{array}
\right).
\end{equation*}

To solve the problem by applying Theorem~\ref{T-AxeqBy}, we have to examine row-monomial matrices $\bm{A}_{1}$ and $\bm{B}_{1}$ formed by fixing one nonzero entry in each rows of the matrices $\bm{A}$ and $\bm{B}$. For each selection of the nonzero entries, we verify the condition at \eqref{I-TrAA1BB1eq1}. A pair of matrices $\bm{A}_{1}$ and $\bm{B}_{1}$ is accepted to construct a solution according to \eqref{E-x-y} if the condition holds, or this pair is rejected otherwise.

Let us consider row-monomial matrices obtained from $\bm{A}$ and $\bm{B}$ in the from
\begin{equation*}
\bm{A}_{1}
=
\left(
\begin{array}{ccc}
3 & \mathbb{0} & \mathbb{0}
\\
\mathbb{0} & 1 & \mathbb{0}
\\
\mathbb{0} & \mathbb{0} & 2
\end{array}
\right),
\qquad
\bm{B}_{1}
=
\left(
\begin{array}{cc}
\mathbb{0} & 1
\\
3 & \mathbb{0}
\\
3 & \mathbb{0}
\end{array}
\right).
\end{equation*}

We start with calculating the matrices
\begin{gather*}
\bm{A}_{1}^{-}\bm{B}
=
\left(
\begin{array}{rr}
-2 & -2
\\
2 & 1
\\
1 & -1
\end{array}
\right),
\qquad
\bm{B}_{1}^{-}\bm{A}
=
\left(
\begin{array}{rrr}
-2 & -2 & -1
\\
2 & \mathbb{0} & -1
\end{array}
\right),
\\
\bm{A}\bm{A}_{1}^{-}
=
\left(
\begin{array}{rcr}
0 & \mathbb{0} & -2
\\
-2 & 0 & -2
\\
\mathbb{0} & 0 & 0
\end{array}
\right),
\qquad
\bm{B}\bm{B}_{1}^{-}
=
\left(
\begin{array}{rrr}
0 & -2 & -2
\\
1 & 0 & 0
\\
0 & 0 & 0
\end{array}
\right),
\end{gather*}
which are then used to obtain the matrices 
\begin{gather*}
\bm{A}_{1}^{-}\bm{B}
\bm{B}_{1}^{-}\bm{A}
=
\left(
\begin{array}{crr}
0 & -4 & -3
\\
3 & 0 & 1
\\
1 & -1 & 0
\end{array}
\right),
\qquad
\bm{B}_{1}^{-}\bm{A}
\bm{A}_{1}^{-}\bm{B}
=
\left(
\begin{array}{cr}
0 & -1
\\
0 & 0
\end{array}
\right),
\\
\bm{A}\bm{A}_{1}^{-}
\bm{B}\bm{B}_{1}^{-}
=
\left(
\begin{array}{crr}
0 & -2 & -2
\\
1 & 0 & 0
\\
1 & 0 & 0
\end{array}
\right).
\end{gather*}

After evaluation of the trace function, we obtain $\operatorname{Tr}(\bm{A}\bm{A}_{1}^{-}\bm{B}\bm{B}_{1}^{-})=0=\mathbb{1}$, and therefore conclude that the equation has regular solutions.

To represent the solution in the form offered by Theorem~\ref{T-AxeqBy}, we first evaluate the matrices
\begin{gather*}
(\bm{A}_{1}^{-}\bm{B}
\bm{B}_{1}^{-}\bm{A})^{2}
=
\left(
\begin{array}{crr}
0 & -4 & -3
\\
3 & 0 & 1
\\
2 & -1 & 0
\end{array}
\right),
\qquad
(\bm{A}_{1}^{-}\bm{B}
\bm{B}_{1}^{-}\bm{A})^{\ast}
=
\left(
\begin{array}{crr}
0 & -4 & -3
\\
3 & 0 & 1
\\
2 & -1 & 0
\end{array}
\right),
\\
\bm{B}_{1}^{-}\bm{A}
(\bm{A}_{1}^{-}\bm{B}
\bm{B}_{1}^{-}\bm{A})^{\ast}
=
\left(
\begin{array}{crr}
1 & -2 & -1
\\
2 & -2 & -1
\end{array}
\right),
\end{gather*}
and then the matrices
\begin{equation*}
(\bm{B}_{1}^{-}\bm{A}
\bm{A}_{1}^{-}\bm{B})^{\ast}
=
\left(
\begin{array}{cr}
0 & -1
\\
0 & 0
\end{array}
\right),
\qquad
\bm{A}_{1}^{-}\bm{B}
(\bm{B}_{1}^{-}\bm{A}
\bm{A}_{1}^{-}\bm{B})^{\ast}
=
\left(
\begin{array}{rr}
-2 & -2
\\
2 & 1
\\
1 & 0
\end{array}
\right).
\end{equation*}

With the matrices obtained and nonzero vectors of parameters denoted by $\bm{u}=(u_{1},u_{2},u_{3})^{T}$ and $\bm{v}=(v_{1},v_{2})^{T}$, the solution given by \eqref{E-x-y} becomes 
\begin{align*}
\bm{x}
&=
\left(
\begin{array}{crr}
0 & -4 & -3
\\
3 & 0 & 1
\\
2 & -1 & 0
\end{array}
\right)\bm{u}
\oplus
\left(
\begin{array}{rr}
-2 & -2
\\
2 & 1
\\
1 & 0
\end{array}
\right)\bm{v},
\\
\bm{y}
&=
\left(
\begin{array}{crr}
1 & -2 & -1
\\
2 & -2 & -1
\end{array}
\right)\bm{u}
\oplus
\left(
\begin{array}{cr}
0 & -1
\\
0 & 0
\end{array}
\right)\bm{v}.
\end{align*}

Furthermore, we can simplify the obtained solution by taking into account the decomposition of matrices 
\begin{gather*}
\left(
\begin{array}{crr}
0 & -4 & -3
\\
3 & 0 & 1
\\
2 & -1 & 0
\end{array}
\right)
=
\left(
\begin{array}{rr}
-2 & -2
\\
2 & 1
\\
1 & 0
\end{array}
\right)
\left(
\begin{array}{crr}
\mathbb{0} & -2 & -1
\\
2 & \mathbb{0} & \mathbb{0}
\end{array}
\right),
\\
\left(
\begin{array}{crr}
1 & -2 & -1
\\
2 & -2 & -1
\end{array}
\right)
=
\left(
\begin{array}{cr}
0 & -1
\\
0 & 0
\end{array}
\right)
\left(
\begin{array}{crr}
\mathbb{0} & -2 & -1
\\
2 & \mathbb{0} & \mathbb{0}
\end{array}
\right).
\end{gather*}

By applying this decomposition, we introduce a new nonzero vector of parameters $\bm{t}=(t_{1},t_{2})^{T}$ given by
\begin{equation*}
\bm{t}
=
\left(
\begin{array}{crr}
\mathbb{0} & -2 & -1
\\
2 & \mathbb{0} & \mathbb{0}
\end{array}
\right)
\bm{u}
\oplus
\bm{v},
\end{equation*}
and then represent the solution in the form
\begin{equation*}
\bm{x}
=
\left(
\begin{array}{rr}
-2 & -2
\\
2 & 1
\\
1 & 0
\end{array}
\right)\bm{t},
\qquad
\bm{y}
=
\left(
\begin{array}{cr}
0 & -1
\\
0 & 0
\end{array}
\right)\bm{t}.
\end{equation*}

It is not difficult to verify, that with $\bm{t}=(1,2)^{T}$, we have the solution
\begin{equation*}
\bm{x}
=
\left(
\begin{array}{c}
0
\\
3
\\
2
\end{array}
\right),
\qquad
\bm{y}
=
\left(
\begin{array}{c}
1
\\
2
\end{array}
\right),
\end{equation*}
which coincides with the solution found in \cite{Cuninghamegreen2003Equation}.

Suppose now that a different choice of nonzero entries in the matrices $\bm{A}$ and $\bm{B}$ is made, which results in the matrices
\begin{equation*}
\bm{A}_{1}
=
\left(
\begin{array}{ccc}
3 & \mathbb{0} & \mathbb{0}
\\
1 & \mathbb{0} & \mathbb{0}
\\
\mathbb{0} & \mathbb{0} & 2
\end{array}
\right),
\qquad
\bm{B}_{1}
=
\left(
\begin{array}{cc}
1 & \mathbb{0}
\\
3 & \mathbb{0}
\\
3 & \mathbb{0}
\end{array}
\right).
\end{equation*}

We calculate the following matrices:
\begin{gather*}
\bm{A}\bm{A}_{1}^{-}
=
\left(
\begin{array}{rcr}
0 & 2 & -2
\\
-2 & 0 & -2
\\
\mathbb{0} & \mathbb{0} & 0
\end{array}
\right),
\qquad
\bm{B}\bm{B}_{1}^{-}
=
\left(
\begin{array}{crr}
0 & -2 & -2
\\
2 & 0 & 0
\\
2 & 0 & 0
\end{array}
\right),
\\
\bm{A}\bm{A}_{1}^{-}
\bm{B}\bm{B}_{1}^{-}
=
\left(
\begin{array}{rcr}
4 & 2 & 2
\\
2 & 0 & 0
\\
2 & 0 & 0
\end{array}
\right).
\end{gather*}

The evaluation of the trace function yields $\operatorname{Tr}(\bm{A}\bm{A}_{1}^{-}\bm{B}\bm{B}_{1}^{-})=4>0=\mathbb{1}$, which means that in this case, the existence condition at \eqref{I-TrAA1BB1eq1} is not fulfilled, and hence the matrices $\bm{A}_{1}$ and $\bm{B}_{1}$ cannot provide a solution to equation \eqref{L-AxeqBy}.

\section{Conclusion}
\label{S-C}
We examined a two-sided vector equation defined in the framework of tropical algebra, where two unknown vectors appear on both sides of the equation. An approach has been proposed to represent an existence condition for solution of the equation and the solution itself in a compact vector form. According to the result of Lemma~\ref{L-AxeqBy}, each solution of the equation is determined by a pair of proper row-monomial matrices $\bm{A}_{1}$ and $\bm{B}_{1}$, for which the existence condition given by Theorem~\ref{T-AxeqBy} holds. Since these matrices are directly used to form generating matrices for solutions, all solutions of the two-sided equation can be obtained by finding all pairs of proper matrices $\bm{A}_{1}$ and $\bm{B}_{1}$.

Note however that the number of pairs of row-monomial matrices to examine grows exponentially with the number of scalar equations in \eqref{L-AxeqBy} (and polynomially with the dimensions of the unknown vectors), which may lead to unreasonably long solution time. Therefore, further development of an economical computational scheme to address this issue is of great importance. As an possible solution, one can consider the backtracking algorithms proposed to improve efficiency of matrix sparsification techniques in \cite{Krivulin2018Solution,Krivulin2020Complete}. Another approach may consist in successive examination of a series of systems that include from $1$ to $m$ scalar equations from \eqref{L-AxeqBy} to reduce the computational time by early detection of inappropriate selection of nonzero elements that violate condition \eqref{I-TrAA1BB1eq1}.

\subsubsection*{Acknowledgments}
The author is very grateful to three anonymous reviewers for their valuable comments and suggestions, which have been incorporated in the revised manuscript.

\bibliographystyle{abbrvurl}

\bibliography{Using_matrix_sparsification_to_solve_tropical_linear_vector_equations}

\begin{thebibliography}{10}

\bibitem{Akian2012Tropical}
M.~Akian, S.~Gaubert, and A.~Guterman.
\newblock Tropical polyhedra are equivalent to mean payoff games.
\newblock {\em Int. J. Algebra and Comput.}, 22(1):{1250001--1}--{1250001--43},
  2012.
\newblock \href {https://doi.org/10.1142/S0218196711006674}
  {\path{doi:10.1142/S0218196711006674}}.

\bibitem{Allamigeon2013Computing}
X.~Allamigeon, S.~Gaubert, and E.~Goubault.
\newblock Computing the vertices of tropical polyhedra using directed
  hypergraphs.
\newblock {\em Discrete Comput. Geom.}, 49(2):247--279, 2013.
\newblock \href {https://doi.org/10.1007/s00454-012-9469-6}
  {\path{doi:10.1007/s00454-012-9469-6}}.

\bibitem{Butkovic1978Certain}
P.~Butkovi\v{c}.
\newblock On certain properties of the systems of linear extremal equations.
\newblock {\em Ekonom.-Mat. Obzor}, 14(1):72--78, 1978.

\bibitem{Butkovic1981Solution}
P.~Butkovi\v{c}.
\newblock Solution of systems of linear extremal equations.
\newblock {\em Ekonom.-Mat. Obzor}, 17(4):402--416, 1981.

\bibitem{Butkovic1984Properties}
P.~Butkovi\v{c}.
\newblock On properties of solution sets of extremal linear programs.
\newblock In R.~E. Burkard, R.~A. Cuninghame-Green, and U.~Zimmermann, editors,
  {\em Algebraic and Combinatorial Methods in Operations Research}, volume~95
  of {\em North-Holland Mathematics Studies}, pages 41--54. North-Holland,
  1984.
\newblock \href {https://doi.org/10.1016/S0304-0208(08)72952-9}
  {\path{doi:10.1016/S0304-0208(08)72952-9}}.

\bibitem{Butkovic2010Maxlinear}
P.~Butkovi\v{c}.
\newblock {\em Max-linear Systems}.
\newblock Springer Monographs in Mathematics. Springer, London, 2010.
\newblock \href {https://doi.org/10.1007/978-1-84996-299-5}
  {\path{doi:10.1007/978-1-84996-299-5}}.

\bibitem{Butkovic1984Elimination}
P.~Butkovi\v{c} and G.~Heged\"{u}s.
\newblock An elimination method for finding all solutions of the system of
  linear equations over an extremal algebra.
\newblock {\em Ekonom.-Mat. Obzor}, 20(2):203--215, 1984.

\bibitem{Butkovic2006Strongly}
P.~Butkovi\v{c} and K.~Zimmermann.
\newblock A strongly polynomial algorithm for solving two-sided linear systems
  in max-algebra.
\newblock {\em Discrete Appl. Math.}, 154(3):437--446, 2006.
\newblock \href {https://doi.org/10.1016/j.dam.2005.09.008}
  {\path{doi:10.1016/j.dam.2005.09.008}}.

\bibitem{Cuninghamegreen2003Equation}
R.~A. Cuninghame-Green and P.~Butkovic.
\newblock The equation ${A}\otimes x={B}\otimes y$ over (max,+).
\newblock {\em Theoret. Comput. Sci.}, 293(1):3--12, 2003.
\newblock \href {https://doi.org/10.1016/S0304-3975(02)00228-1}
  {\path{doi:10.1016/S0304-3975(02)00228-1}}.

\bibitem{Cuninghamegreen2001Equation}
R.~A. Cuninghame-Green and K.~Zimmermann.
\newblock Equation with residuated functions.
\newblock {\em Comment. Math. Univ. Carolin.}, 42(4):729--740, 2001.

\bibitem{Develin2004Tropical}
M.~Develin and B.~Sturmfels.
\newblock Tropical convexity.
\newblock {\em Doc. Math.}, 9:1--27, 2004.

\bibitem{Gaubert2009Tropical}
S.~Gaubert and R.~D. Katz.
\newblock The tropical analogue of polar cones.
\newblock {\em Linear Algebra Appl.}, 431(5):608--625, 2009.
\newblock \href {https://doi.org/10.1016/j.laa.2009.03.012}
  {\path{doi:10.1016/j.laa.2009.03.012}}.

\bibitem{Gaubert2012Tropical}
S.~Gaubert, R.~D. Katz, and S.~Sergeev.
\newblock Tropical linear-fractional programming and parametric mean payoff
  games.
\newblock {\em J. Symbolic Comput.}, 47(12):1447--1478, 2012.
\newblock \href {https://doi.org/10.1016/j.jsc.2011.12.049}
  {\path{doi:10.1016/j.jsc.2011.12.049}}.

\bibitem{Golan2003Semirings}
J.~S. Golan.
\newblock {\em Semirings and Affine Equations Over Them}, volume 556 of {\em
  Mathematics and Its Applications}.
\newblock Kluwer Acad. Publ., Dordrecht, 2003.
\newblock \href {https://doi.org/10.1007/978-94-017-0383-3}
  {\path{doi:10.1007/978-94-017-0383-3}}.

\bibitem{Gondran2008Graphs}
M.~Gondran and M.~Minoux.
\newblock {\em Graphs, Dioids and Semirings}, volume~41 of {\em Operations
  Research/ Computer Science Interfaces}.
\newblock Springer, New York, 2008.
\newblock \href {https://doi.org/10.1007/978-0-387-75450-5}
  {\path{doi:10.1007/978-0-387-75450-5}}.

\bibitem{Heidergott2006Maxplus}
B.~Heidergott, G.~J. Olsder, and J.~{van der Woude}.
\newblock {\em Max {P}lus at work}.
\newblock Princeton series in applied mathematics. Princeton Univ. Press,
  Princeton, NJ, 2006.

\bibitem{Jones2019Twosided}
D.~Jones.
\newblock On two-sided max-linear equations.
\newblock {\em Discrete Appl. Math.}, 254(3):146--160, 2019.
\newblock \href {https://doi.org/10.1016/j.dam.2018.06.011}
  {\path{doi:10.1016/j.dam.2018.06.011}}.

\bibitem{Krivulin2015Extremal}
N.~Krivulin.
\newblock Extremal properties of tropical eigenvalues and solutions to tropical
  optimization problems.
\newblock {\em Linear Algebra Appl.}, 468:211--232, 2015.
\newblock \href {https://doi.org/10.1016/j.laa.2014.06.044}
  {\path{doi:10.1016/j.laa.2014.06.044}}.

\bibitem{Krivulin2015Multidimensional}
N.~Krivulin.
\newblock A multidimensional tropical optimization problem with a non-linear
  objective function and linear constraints.
\newblock {\em Optimization}, 64(5):1107--1129, 2015.
\newblock \href {https://doi.org/10.1080/02331934.2013.840624}
  {\path{doi:10.1080/02331934.2013.840624}}.

\bibitem{Krivulin2015Soving}
N.~Krivulin.
\newblock Solving a tropical optimization problem via matrix sparsification.
\newblock In W.~Kahl, M.~Winter, and J.~N. Oliveira, editors, {\em Relational
  and Algebraic Methods in Computer Science}, volume 9348 of {\em Lecture Notes
  in Comput. Sci.}, pages 326--343. Springer, Cham, 2015.
\newblock \href {https://doi.org/10.1007/978-3-319-24704-5\_20}
  {\path{doi:10.1007/978-3-319-24704-5\_20}}.

\bibitem{Krivulin2017Algebraic}
N.~Krivulin.
\newblock Algebraic solution of tropical optimization problems via matrix
  sparsification with application to scheduling.
\newblock {\em J. Log. Algebr. Methods Program.}, 89:150--170, 2017.
\newblock \href {https://arxiv.org/abs/1504.02602} {\path{arXiv:1504.02602}},
  \href {https://doi.org/10.1016/j.jlamp.2017.03.004}
  {\path{doi:10.1016/j.jlamp.2017.03.004}}.

\bibitem{Krivulin2017Complete}
N.~Krivulin.
\newblock Complete solution of an optimization problem in tropical semifield.
\newblock In P.~H{\"o}fner, D.~Pous, and G.~Struth, editors, {\em Relational
  and Algebraic Methods in Computer Science}, volume 10226 of {\em Lecture
  Notes in Comput. Sci.}, pages 226--241. Springer, Cham, 2017.
\newblock \href {https://doi.org/10.1007/978-3-319-57418-9\_14}
  {\path{doi:10.1007/978-3-319-57418-9\_14}}.

\bibitem{Krivulin2018Complete}
N.~Krivulin.
\newblock Complete algebraic solution of multidimensional optimization problems
  in tropical semifield.
\newblock {\em J. Log. Algebr. Methods Program.}, 99:26--40, 2018.
\newblock \href {https://arxiv.org/abs/1706.00643} {\path{arXiv:1706.00643}},
  \href {https://doi.org/10.1016/j.jlamp.2018.05.002}
  {\path{doi:10.1016/j.jlamp.2018.05.002}}.

\bibitem{Krivulin2020Complete}
N.~Krivulin.
\newblock Complete solution of tropical vector inequalities using matrix
  sparsification.
\newblock {\em Appl. Math.}, 65(6):755--775, 2020.
\newblock \href {https://arxiv.org/abs/2001.07806} {\path{arXiv:2001.07806}},
  \href {https://doi.org/10.21136/AM.2020.0376-19}
  {\path{doi:10.21136/AM.2020.0376-19}}.

\bibitem{Krivulin2023Solution}
N.~K. Krivulin.
\newblock On the solution of a two-sided vector equation in tropical algebra.
\newblock {\em Vestnik St. Petersburg Univ. Math.}, 56(2):172--181, 2023.
\newblock \href {https://doi.org/10.1134/S1063454123020103}
  {\path{doi:10.1134/S1063454123020103}}.

\bibitem{Krivulin2018Solution}
N.~K. Krivulin and V.~N. Sorokin.
\newblock Solution of a multidimensional tropical optimization problem using
  matrix sparsification.
\newblock {\em Vestnik St. Petersburg Univ. Math.}, 51(1):66--76, 2018.
\newblock \href {https://doi.org/10.3103/S1063454118010065}
  {\path{doi:10.3103/S1063454118010065}}.

\bibitem{Lorenzo2011Algorithm}
E.~Lorenzo and M.~J. {de la Puente}.
\newblock An algorithm to describe the solution set of any tropical linear
  system ${A}\odot x={B}\odot x$.
\newblock {\em Linear Algebra Appl.}, 435(4):884--901, 2011.
\newblock \href {https://doi.org/10.1016/j.laa.2011.02.014}
  {\path{doi:10.1016/j.laa.2011.02.014}}.

\bibitem{Maclagan2015Introduction}
D.~Maclagan and B.~Sturmfels.
\newblock {\em Introduction to Tropical Geometry}, volume 161 of {\em Graduate
  Studies in Mathematics}.
\newblock AMS, Providence, RI, 2015.
\newblock \href {https://doi.org/10.1090/gsm/161} {\path{doi:10.1090/gsm/161}}.

\bibitem{Sakarovitch2009Elements}
J.~Sakarovitch.
\newblock {\em Elements of Automata Theory}.
\newblock Cambridge University Press, Cambridge, 2009.
\newblock \href {https://doi.org/10.1017/CBO9781139195218}
  {\path{doi:10.1017/CBO9781139195218}}.

\bibitem{Truffet2010Decomposition}
L.~Truffet.
\newblock A decomposition formula of idempotent polyhedral cones based on
  idempotent superharmonic spaces.
\newblock {\em Beitr. Algebra Geom.}, 51(2):313--336, 2010.

\bibitem{Walkup1998General}
E.~A. Walkup, G.~Borriello, J.~M. Taylor, and M.~Atiyah.
\newblock A general linear max-plus solution technique.
\newblock In J.~Gunawardena, editor, {\em Idempotency}, Publications of the
  Newton Institute, page 406–415. Cambridge Univ. Press, 1998.
\newblock \href {https://doi.org/10.1017/CBO9780511662508.024}
  {\path{doi:10.1017/CBO9780511662508.024}}.

\bibitem{Zimmermann1977General}
K.~Zimmermann.
\newblock A general separation theorem in extremal algebras.
\newblock {\em Ekonom.-Mat. Obzor}, 13(2):179–201, 1977.

\end{thebibliography}

\end{document}